\documentclass[11pt, notitlepage]{article}
\usepackage{amssymb,amsmath,comment}
\usepackage{amsthm}
\catcode`\@=11 \@addtoreset{equation}{section}
\def\thesection{\arabic{section}}

\def\theequation{\thesection.\arabic{equation}}
\catcode`\@=12
\usepackage{colortbl}
\usepackage{hyperref}
\usepackage[mathscr]{eucal}
\usepackage{epsf}
\usepackage{esint}
\usepackage{a4wide}
\def\R{\mathbb{R}}

\newcommand{\De} {\Delta}

\newcommand{\noi} {\noindent}

\newcommand{\mb} {\mathbb}

\usepackage[all]{xy}
\catcode`\@=11
\def\theequation{\@arabic{\c@section}.\@arabic{\c@equation}}
\catcode`\@=12

\def\N{{I\!\!N}}

\newtheorem{Theorem}{Theorem}[section]
\newtheorem{Lemma}[Theorem]{Lemma}

\newtheorem{Remark}[Theorem]{Remark}
\newtheorem{Definition}[Theorem]{Definition}

\begin{document}

{\vspace{0.01in}}

\title{Existence Theory for a class of semilinear mixed local and nonlocal equations involving variable singularities and singular measures}

\author{Sanjit Biswas\footnote{Department of Mathematics and Statistics, Indian Institute of Technology Kanpur, Kanpur-208016, Uttar Pradesh, India, Email: sanjitbiswas410@gmail.com } \,\,and Prashanta Garain\footnote{Department of Mathematical Sciences, Indian Institute of Science Education and Research, Berhampur, 760010, Odisha, India, Email: pgarain92@gmail.com}}

\maketitle

\begin{abstract}\noindent
This article establishes the existence of weak solutions for a class of mixed local-nonlocal problems with pure and perturbed singular nonlinearities. A key novelty is the treatment of variable singular exponents alongside measure-valued data. Notably, both source terms may be measures, with the singular component modeled by both a singular and non-singular measure. Our main focus is on the singular measure data, which appears to be new, even for constant exponents.
\end{abstract}

\maketitle

\noi {Keywords: Mixed local-nonlocal singular problem, existence, variable exponent, measure data.}

\noi{\textit{2020 Mathematics Subject Classification: 35M10, 35M12, 35J75, 35R06, 35R11}

\bigskip

\tableofcontents
\section{Introduction}
In this article, we investigate the existence of weak solutions to a mixed local-nonlocal elliptic problem involving measure data and a variable singular exponent. The problem is formulated as follows:
\begin{align}\label{ME}
     \begin{cases}
        &\mathcal{M}u:=-\mathcal{A}u+\mathcal{B}u=\frac{\nu}{u^{\delta(x)}}+\mu \text{ in } \Omega,\\
        &u=0 \text{ in } \mathbb{R}^N\setminus \Omega \text{ and } u>0 \text{ in }\Omega,
    \end{cases}
\end{align}
where $\Omega\subset\mathbb{R}^N,\,N>2$ is a bounded domain with Lipschitz boundary. The local operator $\mathcal{A}$ is given by $\mathcal{A}u=\text{div}(A(x)\nabla u)$, where $A:\Omega\to\mathbb{R}^{N^2}$ is a bounded elliptic matrix satisfying
\begin{equation}\label{lkernel}
\alpha|\xi|^2\leq A(x)\xi\cdot\xi,\quad |A(x)|\leq\beta,
\end{equation}
for every $\xi\in\mathbb{R}^N$ and for almost every $x\in\Omega$, with constants $0<\alpha\leq\beta$. The nonlocal part $\mathcal{B}$ is defined as
$$
\mathcal{B}u=\text{P.V.}\int_{\mathbb{R}^N}(u(x)-u(y))K(x,y)\,dy,
$$
where P.V. denotes the principal value and $K(x,y)$ is a symmetric kernel satisfying
\begin{equation}\label{nkernel}
\frac{\Lambda^{-1}}{|x-y|^{N+2s}}\leq K(x,y)\leq\frac{\Lambda}{|x-y|^{N+2s}}
\end{equation}
for some constant $\Lambda\geq 1$ and $0<s<1$. A special case arises when $A(x)=I$ and $K(x,y)=|x-y|^{-N-2s}$, reducing $\mathcal{A}$ and $\mathcal{B}$ to the classical Laplacian $-\Delta$ and the fractional Laplacian $(-\Delta)^s$, respectively. In this setting, the operator $\mathcal{M}$ simplifies to the mixed local-nonlocal Laplace operator $-\Delta+(-\Delta)^s$, and equation \eqref{ME} becomes a generalization of the classical mixed singular problem:
\begin{align}\label{ME1}
     \begin{cases}
        &-\Delta u+(-\Delta)^s u=\frac{\nu}{u^\delta(x)}+\mu \text{ in } \Omega,\\
        &u=0 \text{ in } \mathbb{R}^N\setminus \Omega \text{ and } u>0 \text{ in }\Omega.
    \end{cases}
\end{align}
In this work, we consider $\mu$ and $\nu$ to be non-negative bounded Radon measures on a domain $\Omega$. Additionally, we assume that the function $\delta : \overline{\Omega} \to (0, \infty)$ is continuous. The precise assumptions are detailed in Theorems \ref{Theorem1} and \ref{Theorem2}. The strict positivity of $\delta(x)$ induces a singular behavior in the nonlinearity of \eqref{ME} near zero, a hallmark of singular elliptic problems. Consequently, problem \eqref{ME} captures a broad class of mixed singular equations involving both constant and variable singular exponents, as well as measure data.

The purely local problem
\begin{align}\label{MEloc1}
     \begin{cases}
        &-\Delta u=\frac{\nu}{u^{\delta(x)}}+\mu \text{ in } \Omega,\\
        &u=0 \text{ on } \partial\Omega \text{ and } u>0 \text{ in }\Omega,
    \end{cases}
\end{align}
and the purely nonlocal problem
\begin{align}\label{MEnonloc1}
     \begin{cases}
        &(-\Delta)^s u=\frac{\nu}{u^{\delta(x)}}+\mu \text{ in } \Omega,\\
        &u=0 \text{ in } \mathbb{R}^N\setminus\Omega \text{ and } u>0 \text{ in }\Omega
    \end{cases}
\end{align}
have been extensively studied in the literature, particularly with regard to existence results for both constant and variable singular exponents $\delta$, under various assumptions on the given data $\nu$ and $\mu$.

Specifically, when $\delta>0$ is a constant, the analysis of equations \eqref{MEloc1} and \eqref{MEnonloc1} can be found in \cite{Arcoyadie, Arcoyana, BocOrs, CRT} and \cite{Sciunzi, Fang}, respectively. In the case where $\delta$ is variable, the relevant studies for the local and nonlocal problems are presented in \cite{Alvesjde, BGM, CMvar, Garainmm} and \cite{GMcpaa, PKvar}, respectively.

It is important to note that in all the aforementioned works, $\nu$ and $\mu$ are assumed to be integrable functions belonging to suitable Lebesgue spaces. The case involving general Radon measures has been addressed for the local problem in \cite{POesaim, OPdie}, and for the nonlocal setting in \cite{Giri}.

In recent years, mixed local-nonlocal singular problems have attracted significant attention in the analysis of partial differential equations. A notable focus has been on the purely singular mixed local-nonlocal problem \eqref{ME1}, particularly in the case where $\nu$ is a non-negative function belonging to an appropriate Lebesgue space and $\mu=0$. For a constant singular exponent $\delta>0$, the existence and regularity of solutions to problem \eqref{ME1} have been investigated in \cite{Arora, Gjgea, Hichem}, while its quasilinear extension has been studied in \cite{PGjms, Guna}. When the singular exponent $\delta$ is allowed to vary spatially, further results concerning existence and regularity have been obtained in \cite{Biroud, biswas2025, GKK} and the references therein.

Moreover, in the presence of both $\nu$ and $\mu$ as non-negative functions in suitable Lebesgue spaces, the perturbed mixed local-nonlocal problem \eqref{ME1} has also been addressed in \cite{Vecchicvpde, Gjgea} and related works, under the assumption that $0<\delta<1$ is a constant. More recently, the case $\delta\geq 1$ has been considered in \cite{BalDas}, where the authors examined both the semilinear and quasilinear formulations of \eqref{ME1}. Moreover, the case of any $\delta>0$ is studied in \cite{GG}.

Most recently, equation \eqref{ME1} has been investigated in \cite{Ghosh} under the assumptions that $\delta>0$ is a constant, $\nu$ is a positive integrable function, and $\mu$ is a non-negative bounded Radon measure on $\Omega$. Additionally, in \cite{BG}, equation \eqref{ME} is studied in the setting where both $\nu$ and $\mu$ are non-negative bounded Radon measures on $\Omega$ and $\delta$ is a variable function. It is worth noting that in \cite{BG}, the measure $\nu$ is assumed to be non-singular. It is worth emphasizing that, in contrast to \cite{BG}, the present work primarily focuses on the setting where the measure $\nu$ is purely singular. However, our results also encompass certain cases where $\nu$ is non-singular, though subject to more restrictive assumptions than those in \cite{BG}. 

In the non-singular case, the approximate solutions admit a uniform lower bound on every compact subset of $\Omega$; see \cite[Lemma A.1, pages 20–21]{BG}. This result relies on the absolutely continuous component of the measure $\nu$. In contrast, the approximation scheme utilized in \cite{BG} fails to provide such a bound in the presence of singular measures.

To address the singular case, we adopt an alternative approximation strategy inspired by \cite{OPdie}, which is based on a monotone approximation of the measure $\nu$, see Lemma \ref{App2}. Establishing the existence of solutions to the resulting approximate problems presents additional difficulties, particularly requiring that $|\nabla\delta|\in L^N(\Omega)$. Once existence is established, we perform a limiting procedure, which relies on several non-trivial a priori estimates. These estimates are obtained through the careful selection of test functions tailored to the approximate problem. However, the same approach works to deal with non-singular case as well, see Theorems \ref{Theorem1} and \ref{Theorem2}.

The development of this approximation framework and the corresponding analytical techniques represent the main contributions of our approach.

The structure of the paper is as follows: Section 2 introduces the functional framework and presents the main results. Section 3 is devoted to the proofs of these results. Finally, in the Appendix (Section 4), we analyze the approximate problem. 

\textbf{Notations:} For the rest of the paper, unless otherwise mentioned, we will use the following notations and assumptions:
\begin{itemize}
    \item For $k,s\in\R$, we define $T_k(s)=\max\{ -k,\, \min\{s,k\}\}$ and $G_k(s)=(|s|-k)^+sgn(s).$
    \item For a measurable set $A\subset\mathbb{R}^N$, $|A|$ denotes the Lebesgue measure of $A$. Moreover, for a function $u:A\to\R$, we define $u^+:=\max\{ u, 0\}$ and $u^-:=\max\{-u, 0\}.$ 

    \item For $\sigma>1$, we define the conjugate exponent of $\sigma$ by $\sigma'=\frac{\sigma}{\sigma-1}$.
  
    \item $C$ denotes a positive constant, whose value may change from line to line or even in the same line.
    
    \item For a measurable function $f$ over a measurable set $S$ and given constants $c,d$, we write $c\leq u\leq d$ in $S$ to mean that $c\leq u\leq d$ a.e. in $S$.
    
   \item $\Omega\subset\mathbb{R}^N$ with $N>2$ be a bounded Lipschitz domain.

   \item For open sets $\omega$ and $\Omega$ of $\mathbb{R}^N$, by the notation $\omega\Subset\Omega$, we mean that $\overline{\omega}$ is a compact subset of $\Omega$.
\end{itemize}

\section{Functional setting and main results}
\subsection{Functional setting}
We begin this section by introducing Marcinkiewicz spaces. For $q\geq 1$, we define Marcinkiewicz space $M^q(\Omega)$ as the set of all measurable functions $u:\Omega\to\R$ such that there exists a constant $C>0$ with the estimate
$$|\{x\in\Omega: |u(x)|>t\}|\leq \frac{C}{t^q}, \text{ for all } t>0.$$
Note that for a bounded domain $\Omega$, it is enough to have this inequality for all $t\geq t_0$ for some $t_0>0.$ The following embeddings are continuous
\begin{align}\label{maremb}
    L^q(\Omega)\hookrightarrow M^q(\Omega)\hookrightarrow L^{q-\eta}(\Omega),
\end{align}
for any {$\eta\in(0,q-1].$} For more details, see \cite{JMar} and the references therein.

The Sobolev space $W^{1,p}(\Omega)$ for $1<p<\infty$, is defined to be the space of functions $u:\Omega\to\mathbb{R}$ in $L^p(\Omega)$ such that the partial derivatives $\frac{\partial u}{\partial x_i}$ for $1\leq i\leq N$ exists in the weak sense and belong to $L^p(\Omega)$. The space $W^{1,p}(\Omega)$ is a Banach space (see \cite{LC}) equipped with the norm:
$$
\|u\|_{W^{1,p}(\Omega)} = \|u\|_{L^p(\Omega)} + \|\nabla u\|_{L^p(\Omega)},
$$
where $\nabla u=\Big(\frac{\partial u}{\partial x_1},\ldots,\frac{\partial u}{\partial x_N}\Big)$.  
The fractional Sobolev space $W^{s,p}(\Omega)$ for $0<s<1<p<\infty$, is defined by
$$
W^{s,p}(\Omega)=\Bigg\{{u:\Omega\to\mathbb{R}:\,}u\in L^p(\Omega),\,\frac{|u(x)-u(y)|}{|x-y|^{\frac{N}{p}+s}}\in L^p(\Omega\times \Omega)\Bigg\}
$$
under the norm
$$
\|u\|_{W^{s,p}(\Omega)}=\left(\int_{\Omega}|u(x)|^p\,dx+\int_{\Omega}\int_{\Omega}\frac{|u(x)-u(y)|^p}{|x-y|^{N+ps}}\,dx\,dy\right)^\frac{1}{p}.
$$
We refer to \cite{Hitchhikersguide} and the references therein for more details on fractional Sobolev spaces. Due to the mixed behavior of our equations, following \cite{Vecchihong, VecchiBO, Vecchihenon}, we consider the space
$$
W_0^{1,p}(\Omega)=\{u\in W^{1,p}(\mathbb{R}^N):u=0\text{ in }\mathbb{R}^N\setminus\Omega\}
$$
under the norm
$$
\|u\|_{W_0^{1,p}(\Omega)}=\left(\int_{\Omega}|\nabla u|^p\,dx+\int_{\mathbb{R}^{N}}\int_{\mathbb{R}^{N}}\frac{|u(x)-u(y)|^p}{|x-y|^{N+ps}}\, dx dy\right)^\frac{1}{p}.
$$
{Using Lemma \ref{locnon1} below, we observe that the norm $\|u\|_{W_0^{1,p}(\Omega)}$ defined above is equivalent to the norm $\|u\|=\|\nabla u\|_{L^p(\Omega)}$.} Let $0<s\leq 1<p<\infty$. Then we say that $u\in W^{s,p}_{\mathrm{loc}}(\Omega)$ if $u\in W^{s,p}(K)$ for every $K\Subset\Omega$.\\
The following result can be found in \cite[Lemma $2.1$]{Silva}.
\begin{Lemma}\label{locnon1}
Let $0<s<1<p<\infty$. There exists a constant $C=C(N,p,s,\Omega)>0$ such that
\begin{equation}\label{locnonsem}
\int_{\mathbb{R}^N}\int_{\mathbb{R}^N}\frac{|u(x)-u(y)|^p}{|x-y|^{N+ps}}\,dx\,dy\leq C\int_{\Omega}|\nabla u|^p\,dx
\end{equation}
for every $u\in W_0^{1,p}(\Omega)$.
\end{Lemma}

For the subsequent Sobolev embedding, refer to \cite{LC}, for instance.
\begin{Lemma}\label{emb}
Let $1<p<\infty$. Then the embedding operators
\[
W_0^{1,p}(\Omega)\hookrightarrow
\begin{cases}
L^t(\Omega),&\text{ for }t\in[1,p^{*}],\text{ if }1<p<N,\\
L^t(\Omega),&\text{ for }t\in[1,\infty),\text{ if }p=N,\\
L^\infty(\Omega),&\text{ if }p>N
\end{cases}
\]
are continuous. Moreover, they are compact except for $t=p^*$ if $1<p<N$. Here $p^*=\frac{Np}{N-p}$ if $1<p<N$.
\end{Lemma}
We now briefly discuss some preliminary results related to measures (see \cite{GOB, MD, OPdie}) and $p$-capacity (see \cite{Heinonen}). {We define $M(\Omega)$ as a set of all signed Radon measures on $\Omega$ with bounded total variation (as usual, identified with a linear map $u\to \int_\Omega u\, d\mu$ on $C(\overline{\Omega})$). If $\nu\in M(\Omega)$ is a non-negative Radon measure then by the Lebesgue's decomposition theorem \cite[page 384]{Royden} and the Radon-Nikodym theorem \cite[page 382]{Royden}, we have $$\nu=\nu_a+\nu_s,$$  where $\nu_a$ is a Lebesgue measurable function and $\nu_s$ is singular with respect to the Lebesgue measure, i.e., $\nu_s$ is concentrated on a set of zero Lebesgue measure. Moreover, if $\nu$ is bounded then $\nu_a\in L^1(\Omega)$. If the function $\nu_a$ is not identically zero, then we say that $\nu$ is {non-singular} with respect to the Lebesgue measure, otherwise it is called singular measure.

Suppose $p>1$, then for a compact set $K\subset\Omega$, the $p$-capacity of $K$ is denoted by $\mathrm{cap_{\text{$p$}}(K)}$ and defined as
$$\mathrm{cap\text{$_p$}(K)}:=\inf \left \{ \int_\Omega|\nabla\phi|^p dx : \phi\in C^\infty_0(\Omega), \phi\geq \chi_K \right \},$$
where $$\chi_K(x):=\begin{cases}
    1\text{ if } x\in K,\\
    0 \text{ otherwise.}
\end{cases}$$
Finally, $p$-capacity of any subset $B$ of $\Omega$ is defined by the standard way.} We say a measure $\nu\in M(\Omega)$ is absolutely continuous with respect to $p$-capacity if the following holds: \textit{$\nu(E)=0$ for every $E\subset\Omega$ such that $\mathrm{cap_\text{$p$}(E)}=0$.}
We define 
$$
M^p_0(\Omega)=\{\nu\in M(\Omega):\nu \text{ is absolutely continuous with respect to } p \text{-capacity}\}.
$$
One can observe that if $1<p_1<p_2$, then $M^{p_1}_0(\Omega)\subset M^{p_2}_0(\Omega).$

The following characterization from \cite[Theorem 2.1]{GOB} is very useful for us.
\begin{Theorem}\label{dec1}
    Let $1<p<\infty$ and $\nu\in M^p_0(\Omega)$. Then there exists $f\in L^1(\Omega)$ and $G\in (L^{p'}(\Omega))^N$ such that $\nu=f-\mathrm{div}(G)\in L^1(\Omega)+W^{-1,p'}(\Omega)$ {in $\mathcal{D}'(\Omega)$ (space of distributions).} Furthermore, if $\nu$ is non-negative, then $f$ is non-negative.
\end{Theorem}}
Next, we define the notion of weak solutions of the problem \eqref{ME}, which is well stated as observed in \cite[Remark 2.8]{BG}.
\begin{Definition}\label{def1}
    Let $0<s<1<q<\infty$ and $\delta:\overline{\Omega}\to(0,\infty)$ be a continuous function. {Suppose that $\mu$ and $\nu$ are two non-negative bounded Radon measures on $\Omega$ such that $\nu\in M^q_0(\Omega)$. We say that $u\in W^{1,q}_{loc}(\Omega)\cap L^1(\Omega)$ is a weak solution of the equation (\ref{ME}) if $u=0$ in $\mathbb{R}^N\setminus\Omega$ and 
    \begin{enumerate}
        \item[(a)] for every $\omega\Subset\Omega$, there exists a constant $C(\omega)>0$ such that $u\geq C(\omega)>0$ in $\omega$ and
        
        \item[(b)] for every $\phi\in C^{\infty}_c(\Omega)$, we have
    \begin{align}\label{SC}
          \int_\Omega A(x)\nabla u\cdot\nabla\phi\, dx+\int_{\mathbb{R}^N}\int_{\mathbb{R}^N}&{(u(x)-u(y))(\phi(x)-\phi(y))}K(x,y)\, dx dy\nonumber\\
          &=\int_\Omega\frac{\phi}{u^{\delta(x)}}\, d\nu+\int_\Omega \phi\,  d\mu.
    \end{align}
     \end{enumerate}
    }
\end{Definition}

The following result can be found in \cite{meapp, DAL}.

\begin{Lemma}\label{App2}
    Suppose $1<p<\infty$ and $\nu\in M^p_0(\Omega)$ is a non-negative measure.  Then, there exists a non-negative increasing sequence of measures $\{\nu_n\}_{n\in\mb N}\subset W^{-1,p'}(\Omega)$ such that $\nu_n\to\nu$ in $\mathcal{M}(\Omega)$ (total variation norm), i.e., $|\nu_n-\nu|(\Omega)\to0$ as $n\to\infty$.
\end{Lemma}
For the next result we refer \cite[Proposition 2.7]{MD}. Before stating the result, we note that it is well-known that for every $u\in W^{1,p}(\Omega)$, there exists $\mathrm{cap}_p$-quasi continuous representative $\Tilde{u}$ of $u$, i.e.,
\begin{itemize}
    \item [(i)] $u=\Tilde{u}$ a.e. in $\Omega$,
    \item[(ii)] for every $\epsilon>0$, there exists a set $S\subset\Omega$ such that $\mathrm{cap}_p(S)<\epsilon$ and $\Tilde{u}$ is continuous in $\Omega\setminus S$. 
\end{itemize}

\begin{Lemma}\label{Bdd}
     Suppose $1<p<\infty$ and $\nu\in M^p_0(\Omega)$ is a non-negative measure. Assume that $u\in W^{1,p}_0(\Omega)\cap L^\infty(\Omega)$ is a non-negative function. Then, up to the choice of its $\mathrm{cap}_p$-quasi continuous representative, $u\in L^\infty(\Omega,\nu)$, and $$\int_\Omega u\;d\nu\leq \|u\|_{L^\infty(\Omega)}\,\nu(\Omega).$$
\end{Lemma}

\begin{Definition}\label{ntop}
A sequence {$\{\mu_n\}_{n\in\mathbb{N}}\subset M(\Omega)$} is said to converge to a measure $\mu\in M(\Omega)$ in narrow topology if for every $\phi\in C^\infty_c(\Omega)$, we have $$\lim_{n\to\infty}\int_\Omega \phi d\mu_n=\int_\Omega \phi d\mu.$$
\end{Definition}

Before stating our main results below, we define the condition $(P_{\epsilon,\delta_*})$ below.\\
\textbf{Condition $(P_{\epsilon,\delta_*})$:} We say that a continuous function $\delta:\overline{\Omega}\to (0,\infty)$, satisfies the condition $(P_{\epsilon,\delta_*})$, if there exist $\delta_*\geq 1$ and $\epsilon>0$ such that $\delta(x)\leq \delta_*$ for every $x\in\Omega_\epsilon$, where 
$
\Omega_\epsilon:=\{y\in\Omega:\text{dist\,}(y,\partial\Omega)<\epsilon\}.
$
\subsection{Main results}
We state our main existence results below.
\begin{Theorem}\label{Theorem1}(Variable singular exponent)
   Let $\delta:\overline{\Omega}\to(0,\infty)$ be a continuous function such that $\delta$ is locally Lipschitz continuous in $\Omega$, $|\nabla\delta|\in L^N(\Omega)$, and $\delta$ satisfies $P_{\epsilon,\delta_*}$ for some $\;\delta_*\geq 1$ and for some $\epsilon>0$.
   
     Suppose $\mu,\;\nu$ are non-negative bounded Radon measures on $\Omega$, with $\nu\in M^p_0(\Omega)\setminus\{0\}$ for some $1<p<\frac{N}{N-1}$. Then the problem $(\ref{ME})$ admits a weak solution $u\in W^{1,p}_{\mathrm{loc}}(\Omega)\cap L^1(\Omega)$ in the sense of Definition \ref{def1} such that
\begin{enumerate}
    \item[(i)]If $\delta_*=1$, then $u\in W^{1,p}_0(\Omega)$.
    \item[(ii)] If $\delta_*>1$, then $T_k(u)\in W^{1,2}_{\mathrm{loc}}(\Omega)$ such that $T_k^\frac{\delta_*+1}{2}(u)\in W^{1,2}_0(\Omega)$ for every $k>0$.
\end{enumerate}
\end{Theorem}
\begin{Remark}
   Furthermore, upon close examination of the proof of Theorem \ref{Theorem1}, it becomes evident that the condition $|\nabla\delta|\in L^N(\Omega)$ plays a crucial role in establishing the existence of the approximate solution as shown in Lemma \ref{Existence}. In contrast, as shown in \cite{BG}, this assumption was not required in the non-singular case. As a result, when $\nu$ is non-singular, Theorem \ref{Theorem1} agrees with \cite[Theorem 2.12]{BG}, provided the additional condition $|\nabla\delta|\in L^N(\Omega)$ holds. On the other hand, Theorem \ref{Theorem1} is completely new when $\nu$ is a singular measure. 
\end{Remark}

When $\delta>0$ is a constant and $\mu$ is a suitable Lebesgue integrable function, we have the following existence result for a more general class of measures $\nu$.

\begin{Theorem}\label{Theorem2}(Constant singular exponent)
    Assume that $\delta:\overline{\Omega}\to(0,\infty)$ is a constant function. We define
    $$p=\begin{cases}
    \frac{N(\delta+1)}{N+\delta-1}, \text{ if }0<\delta<1,\\
    2, \text{ if }\delta\geq 1.
\end{cases}$$
{Let $\mu\in L^\frac{N(\delta+1)}{N+2\delta}(\Omega )$ be a non-negative function in $\Omega$. 
 Further, assume that $\nu\in M^{p}_0(\Omega)\setminus\{0\}$ is a non-negative  bounded Radon measure on $\Omega$.} Then the equation (\ref{ME}) admits a weak solution $u\in W^{1,p}_{\mathrm{loc}}(\Omega)\cap L^1(\Omega)$ in the sense of Definition \ref{def1} such that
\begin{enumerate}
    \item [(i)] If $0<\delta\leq 1$, then $u\in W_0^{1,p}(\Omega)$.
    \item[(ii)] If $\delta>1$, then $u\in W_{\mathrm{loc}}^{1,2}(\Omega)$ such that $u^\frac{\delta+1}{2}\in W_0^{1,2}(\Omega)$. 
\end{enumerate}
\end{Theorem}

\begin{Remark}\label{mrmk}
One can observe along the lines of the proofs of our main existence results that Theorems \ref{Theorem1} and \ref{Theorem2} hold even for the purely local equation that can be obtained by replacing $\mathcal{M}$ with the operator $\mathcal{A}$ in the equation \eqref{ME} where $\mathcal{A}u=\text{div}(A(x)\nabla u)$ with $A:\Omega\to\mathbb{R}^{N^2}$ a bounded elliptic matrix satisfying \eqref{lkernel}. To the best of our knowledge, such results are new even in the purely local case.
\end{Remark}

\section{Proof of the main results}
\subsection{Proof of {Theorem \ref{Theorem1}}}
Suppose that $\delta:\overline{\Omega}\to (0,\infty)$ is a continuous function and satisfies the properties (a), (b), and (c). That is, there exists $\delta_*\geq 1$ and $\epsilon>0$ such that $$\delta(x)\leq \delta_* \text{ in }\Omega_\epsilon.$$
We also assume that $\nu\in M^p_0(\Omega)\setminus\{0\}$ for some $1<p<\frac{N}{N-1}$, is a non-negative bounded Radon measure on $\Omega$ and $\mu$ is a non-negative bounded Radon measure on $\Omega$.

Since $\nu\in M^p_0(\Omega)$, by Theorem \ref{dec1}, there exists $0\leq H\in L^1(\Omega)$ and $G\in (L^{p'}(\Omega))^N$ such that $\nu=H-\mathrm{div}\, G$ (in distributional sense). Furthermore, due to Lemma \ref{App2}, there exists an increasing sequence of non-negative measures $\{\nu_n\}_{n\in\mathbb{N}}\subset W^{-1,p'}(\Omega)\setminus\{0\}$ in $\Omega$ such that $\nu_n\to\nu$ in $\mathcal{M}(\Omega)$ (total variation norm). Since $\mu$ is a non-negative bounded Radon measure on $\Omega$, there exists a {non-negative} sequence $\{g_n\}_{n\in\mathbb{N}}\subset L^\infty(\Omega)$ such that $||g_n||_{L^1(\Omega)}\leq C$ {for some constant $C>0$ independent of $n$} and $g_n\rightharpoonup \mu$ in the narrow topology (\cite[Theorem A.7]{OLP}). For each $n\in\mathbb{N}$, we consider the following approximation of the given equation (\ref{ME}):
\begin{align}\label{DP}
    \begin{cases}
        \mathcal{M}u=\frac{\nu_n}{(u+\frac{1}{n})^{\delta(x)}}+{g_n} \text{ in } \Omega,\\
        u=0 \text{ in }\mathbb{R}^N\setminus\Omega \text{ and } u>0 \text{ in }\Omega.
    \end{cases}
\end{align}
    By Lemmas \ref{Existence}-\ref{LD11}, for each $n\in\mathbb{N}$, there exists a unique weak solution $u_n\in W^{1,2}_0(\Omega)$ to the equation (\ref{DP}) such that for every $\omega\Subset\Omega$, there exists a constant $C(\omega)>0$ {(independent of $n$)} such that $u_n\geq C(\omega)$ in $\omega$ for all $n$. {By the weak formulation of (\ref{DP}), for every $\phi\in C^\infty_c(\Omega)$, we get
\begin{align}\label{A16}
\int_\Omega &A(x)\nabla u_n\cdot\nabla\phi \, dx+\int_{\mathbb{R}^N}\int_{\mathbb{R}^N}(u_n(x)-u_n(y))(\phi(x)-\phi(y))K(x,y)\, dx dy\nonumber\\
&=\int_\Omega \frac{\phi}{(u_n+\frac{1}{n})^{\delta(x)}}\, d\nu_n+\int_\Omega\phi g_n\, dx.
\end{align}
We pass to the limit in \eqref{A16} for the cases $\delta_*=1$ and $\delta_*>1$ below.}
    \begin{enumerate}
\item[$(i)$] {Let $\delta_*=1$. Since $\nu\in M^p_0(\Omega)$ and $0\leq\nu_n\leq\nu$, we have $\nu_n\in M^p_0(\Omega)$. Furthermore, since $1<p<\frac{N}{N-1}$ and $\nu_n\in W^{-1,p'}(\Omega)$, we conclude that the sequence $\{\nu_n\}\subset W^{-1,2}(\Omega)\cap M^2_0(\Omega)$ is uniformly bounded in $\mathcal{M}(\Omega)$. By} Lemma \ref{LD2}-$(a)$, there exists a subsequence of $\{u_n\}_{n\in\mathbb{N}}$, still denoted by {$\{u_n\}_{n\in\mathbb{N}}\subset W_0^{1,p}(\Omega)$} and $u\in W^{1,p}_0(\Omega)$ such that $u_n\rightharpoonup u$ {weakly} in $W^{1,p}_0(\Omega)$ and $u_n\to u$ {pointwise} in $\Omega.$
 Since $u_n\rightharpoonup u$ weakly in  $W^{1,p}_0(\Omega)$, for every $\phi\in C_c^{\infty}(\Omega)$, it follows that 
 \begin{align}\label{A17}
    &\lim_{n\to\infty}\int_\Omega A(x)\nabla u_n\cdot\nabla\phi\, dx=\int_\Omega A(x)\nabla u\cdot\nabla\phi\, dx,
    \end{align}
    and 
    \begin{equation}\label{B00}
    \begin{split}
&\lim_{n\to\infty }\int_{\mathbb{R}^N}\int_{\mathbb{R}^N}(u_n(x)-u_n(y))(\phi(x)-\phi(y))K(x,y)\, dx dy\\
&=\int_{\mathbb{R}^N}\int_{\mathbb{R}^N}(u(x)-u(y))(\phi(x)-\phi(y))K(x,y)\, dx dy.
\end{split}
\end{equation}
Moreover, since $g_n\rightharpoonup\mu$ in {the narrow topology}, for every $\phi\in C_c^{\infty}(\Omega)$, we have
\begin{align}
\lim_{n\to\infty}\int_\Omega \phi g_n\, dx=\int_\Omega \phi d\mu.
\end{align}
{Let $\omega=\mathrm{supp}\,\phi$}, then there exists a constant $C=C(\omega)>0$ {(independent of $n$)} such that $u_n\geq C$ in $\omega$, for all $n$. In order to pass the limit to the integral in the second last term of (\ref{A16}), we break it into two parts as
$$\int_\Omega \frac{\phi}{(u_n+\frac{1}{n})^{\delta(x)}}\, d\nu_n=\int_\Omega \frac{\phi}{(u_n+\frac{1}{n})^{\delta(x)}}\, d\nu-\int_\Omega \frac{\phi}{(u_n+\frac{1}{n})^{\delta(x)}}\, d(\nu-\nu_n).$$
Let us estimate the second term fast as 
\begin{align}\label{Sat}
    \int_\Omega \frac{\phi}{(u_n+\frac{1}{n})^{\delta(x)}}\, d(\nu-\nu_n)&\leq \Big\|\frac{\phi}{(u_n+\frac{1}{n})^{\delta(x)}}\Big\|_{L^\infty(\Omega)}|(\nu-\nu_n)(\Omega)|\nonumber\\
    &\leq\frac{\|\phi\|_{L^\infty(\Omega)}}{\|C^{\delta(x)}\|_{L^\infty(\Omega)}}|(\nu-\nu_n)(\Omega)|\to 0\text{ as }n\to\infty.
\end{align}
Since $\nu=H-\mathrm{div}(G)$, we get
\begin{align}\label{A18}
     \int_\Omega\frac{\phi}{(u_n+\frac{1}{n})^{\delta(x)}}\, d\nu=\int_\Omega\frac{H\phi}{(u_n+\frac{1}{n})^{\delta(x)}}\, dx+\int_\Omega G\cdot \nabla\left(\frac{\phi}{(u_n+\frac{1}{n})^{\delta(x)}}\right)\,dx.
\end{align}
By the Lebesgue's dominated convergence theorem, we obtain
\begin{align}\label{AA2}
    \lim_{n\to\infty}\int_\Omega\frac{H\phi}{(u_n+\frac{1}{n})^{\delta(x)}}\, dx=\int_\Omega\frac{H\phi}{u^{\delta(x)}}\, dx.
\end{align}
Since $\delta$ satisfies (a), along the lines of proof of the estimate of (3.12) on \cite[page 14]{BG}, we obtain
\begin{align}\label{SA2}
    \lim_{n\to\infty}\int_\Omega G\cdot\nabla\left (\frac{\phi}{(u_n+\frac{1}{n})^{\delta(x)}}\right )\, dx=\int_\Omega G\cdot\nabla\left (\frac{\phi}{u^{\delta(x)}}\right )\, dx.
\end{align}
Thus, letting  $n\to\infty$ in both sides of the equality (\ref{A16}) and using (\ref{A17})-(\ref{SA2}), we obtain
\begin{equation*}
\begin{split}
&\int_\Omega A(x)\nabla u\cdot\nabla\phi\, dx+\int_{\mathbb{R}^N}\int_{\mathbb{R}^N}(u(x)-u(y))(\phi(x)-\phi(y))K(x,y)\, dx dy\\
&=\int_\Omega \frac{H\phi}{u^{\delta(x)}}\, dx\nonumber+\int_\Omega G\cdot\nabla\left (\frac{\phi}{u^{\delta(x)}}\right )\, dx+\int_\Omega\phi d\mu\nonumber\\
&=\int_\Omega \frac{\phi}{u^{\delta(x)}}\, d\nu+\int_\Omega \phi d\mu.
\end{split}
\end{equation*}
Hence, $u\in W^{1,p}_0(\Omega)$ is a weak solution of the equation (\ref{ME}).
\item[$(ii)$]
{Let $\delta_*>1$. Since $1<p<\frac{N}{N-1}$,} applying Lemma \ref{LD2}-$(b)$, we conclude that there exists a subsequence of $\{u_n\}_{n\in\mathbb{N}}$, still denoted by {$\{u_n\}_{n\in\mathbb{N}}\subset W^{1,p}_{\mathrm{loc}}(\Omega)$} and $u\in W^{1,p}_{\mathrm{loc}}(\Omega)$ such that 
\begin{align*}
    \begin{cases}
        u_n\rightharpoonup u \text{ weakly in } W^{1,p}_{\mathrm{loc}}(\Omega),\\
        u_n\to u \text{ pointwise in }\Omega.
    \end{cases}
\end{align*}
{In this case by repeating} the similar proof as in $(i)$ above, together with the similar argument as in the proof of the estimate (3.15) on \cite[page 15]{BG}, one can pass to the limit in all the integrals in {(\ref{A16})} and obtain
\begin{align*}
     \int_\Omega A(x)\nabla u\cdot\nabla\phi\, dx+\int_{\mathbb{R}^N}\int_{\mathbb{R}^N}&(u(x)-u(y))(\phi(x)-\phi(y))K(x,y)\, dxdy\nonumber\\
     &=\int_\Omega\frac{\phi}{u^\delta}\, d\nu+\int_\Omega \phi d\mu.
\end{align*}
Thus $u\in W^{1,p}_{\mathrm{loc}}(\Omega)$ solves the equation (\ref{ME}). {Moreover, by Lemma \ref{LD2}-$(b)$ and {Lemma \ref{emb}}, one has the sequence $\{G_1(u_n)\}_{n\in\mathbb{N}}$ is uniformly bounded in $L^1(\Omega)$. Taking this into account along with the fact that $|T_1(u_n)|\leq 1$ in $\Omega$ and $u_n=T_1(u_n)+G_1(u_n)$, we obtain that $\{u_n\}_{n\in\mathbb{N}}$ is uniformly bounded in $L^1(\Omega)$. Thus, by Fatou's Lemma, it follows that $u\in L^1(\Omega).$ Furthermore, using Lemma \ref{LD2}-$(b)$, one obtain that $T_k(u)\in W^{1,2}_{\mathrm{loc}}(\Omega)$ such that $T^\frac{\delta_*+1}{2}_k(u)\in W^{1,2}_0(\Omega)$ for every $k>0$.} 
\end{enumerate}

\subsection{Proof of Theorem \ref{Theorem2}} 
Suppose that $\nu\in M^{p}_0(\Omega)\setminus\{0\}$ is a non-negative bounded Radon measure on $\Omega$, which is singular with respect to the Lebesgue measure, and
 $\{\nu_n\}_{n\in\mathbb{N}}\subset W^{-1,p'}(\Omega)\setminus\{0\}$ is the same sequence discussed in the proof of Theorem \ref{Theorem1}. For each $n\in\mathbb{N}$, we consider the following approximation of the equation (\ref{ME}):
\begin{align}\label{BP}
    \begin{cases}
    \mathcal{M}u=\frac{\nu_n}{(u+\frac{1}{n})^\delta}+T_n(\mu) \text{ in } \Omega,\\
    u=0 \text{ in }\mathbb{R}^N\setminus\Omega \text{ and } u>0 \text{ in }\Omega,
    \end{cases}
\end{align}
where $\mu\in L^\frac{N(\delta+1)}{N+2\delta}(\Omega)$ is a non-negative function in $\Omega$. Finally, we pass to the limit in the weak formulation of \eqref{BP} to conclude the result. This follows by taking into account Lemma \ref{LD11}, Lemma \ref{BL2}-$(a), (b)$ along with Lemma \ref{BL2}-$(c)$ and proceeding along the lines of the proof of Theorem \ref{Theorem1}.

\section{Appendix}
\subsection{Approximate problem}
Throughout this subsection, we assume that $\delta:\overline{\Omega}\to(0,\infty)$ is a continuous function, which is locally Lipschitz continuous in $\Omega$ and  $|\nabla\delta|\in L^N(\Omega)$. Consequently, $|\nabla\delta|\in L^\infty_{\mathrm{loc}}(\Omega)\cap L^N(\Omega)$. We consider the problem
\begin{align}\label{APE0}
    \begin{cases}
        \mathcal{M}u=\frac{\Tilde{\nu}}{(u+\tau)^{\delta(x)}}+g \text{ in }\Omega,\\
        u=0 \text{ in }\mathbb{R}^N\setminus\Omega \text{ and } u>0 \text{ in }\Omega,
    \end{cases}\tag{$P_{\tau,\Tilde{\nu},g}$}
\end{align}
where $\tau>0$ is a constant, $\Tilde{\nu}\in W^{-1,2}(\Omega)\setminus\{0\}$ is a non-negative measure, and $g\in L^2(\Omega)$ is a non-negative function.

\begin{Lemma}\label{Existence}
    There exists a unique weak solution $u\in W^{1,2}_0(\Omega)$ of the problem (\ref{APE0}).
\end{Lemma}
\begin{proof}
Since $\Tilde{\nu}\in W^{-1,2}(\Omega)\setminus\{0\}$, by the Lax-Milgram theorem {\cite[page 315]{LC}}, there exists a unique $w\in W^{1,2}_0(\Omega)\setminus\{0\}$ such that 
\begin{equation}
        -\Delta w=\Tilde{\nu} \text{ in }\Omega,
\end{equation}
Furthermore, by \cite[Lemma 2.1]{MC18}, for each $n\in\N$, there exists a unique $w_n\in W^{1,2}_0(\Omega)$ such that $w_n\neq w_m$ for $n\neq m$ satisfying
\begin{equation}
        -\frac{1}{n}\Delta w_n+w_n=w \text{ in }\Omega,
\end{equation}
with $w_n\to w$ strongly in $W^{1,2}_0(\Omega)$ and observe that $\De w_n\in W^{1,2}_0(\Omega)$. As in \cite[Lemma A.1]{BG}, there exists $u_n\in W^{1,2}_0(\Omega)$ satisfying the equation
\begin{align}\label{APE01}
    \mathcal{M}u_n=\frac{-\De w_n}{(|u_n|+\tau)^{\delta(x)}}+g \text{ in }\Omega.
\end{align}
In order to prove the boundedness of the sequence $\{u_n\}_{n\in\N}$, we incorporate $u_n$ as a test function in the weak formulation of \eqref{APE01} and utilize \eqref{lkernel}, \eqref{nkernel} to obtain
\begin{align}\label{Ine}
   \alpha \int_\Omega|\nabla u_n|^2\;dx&\leq \int_\Omega \frac{u_n(-\De w_n)}{(|u_n|+\tau)^{\delta(x)}}\;dx+\int_\Omega u_ng\;dx\nonumber\\
   &= \int_\Omega \nabla w_n\cdot\nabla\left(\frac{u_n}{(|u_n|+\tau)^{\delta(x)}}\right)\;dx+\int_\Omega u_ng\;dx\nonumber\\
    &\leq \|\nabla w_n\|_{L^2(\Omega)}\Big\|\nabla\left(\frac{u_n}{(|u_n|+\tau)^{\delta(x)}}\right)\Big\|_{L^2(\Omega)}+\|g\|_{L^2(\Omega)}\|u_n\|_{L^2(\Omega)}.
\end{align}
Using Poincar\'{e} inequality, H\"{o}lder inequality, Lemma \ref{emb} and the fact $|\nabla\delta|\in L^N(\Omega)$, we deduce
\begin{align}\label{Ineq2}
    \int_\Omega \Big|\nabla\left(\frac{u_n}{(|u_n|+\tau)^{\delta(x)}}\right)\Big|^2\;dx&\leq C\Big(\int_\Omega\Big|\frac{\nabla u_n}{(|u_n|+\tau)^{\delta(x)}}\Big|^2\;dx+\int_\Omega \Big|\frac{u_n\nabla\delta}{(|u_n|+\tau)^{\delta(x)}}\mathrm{log}(|u_n|+\tau)\Big|^2\;dx\nonumber\\
    &+\int_\Omega \Big|\frac{\delta u_n\nabla u_n}{(|u_n|+\tau)^{\delta(x)+1}}\Big|^2\;dx\Big)\nonumber\\
    &\leq C\Big(\int_\Omega |\nabla u_n|^2\;dx+\int_\Omega |u_n|^2|\nabla\delta|^2\;dx+\int_\Omega |\nabla u_n|^2\;dx\Big)\nonumber\\
    &\leq C\Big (\int_\Omega |\nabla u_n|^2\;dx+\Big(\int_\Omega |u_n|^{2^*}\;dx\Big)^\frac{2}{2^*}\Big(\int_\Omega |\nabla\delta|^N\;dx\Big)^\frac{2}{N}\nonumber\\&+\int_\Omega |\nabla u_n|^2\;dx\Big)
    \leq C\int_\Omega |\nabla u_n|^2\;dx,
\end{align}
where $C>0$ is a constant independent of $n$. In the second line, we used the fact
\begin{align}\label{LC}
   \Big|\frac{\mathrm{log}(|u_n|+\tau)}{(|u_n|+\tau)^{\delta(x)}}\Big|=\Big|\Big(\frac{1}{(|u_n|+\tau)^{\delta(x)-r}}\Big)\Big(\frac{\mathrm{log}(|u_n|+\tau)}{(|u_n|+\tau)^r}\Big)\Big|\leq C\text{ in }\Omega,
\end{align}
where $r:=\inf_{\overline{\Omega}}\delta>0$, and $C>0$ is independent of $n$. Here, we used the boundedness of the function $\frac{\mathrm{log}t}{t^r}$ in $[\tau,\infty)$. Utilizing inequality (\ref{Ineq2}) along with the boundedness of $\{w_n\}_{n\in\mb N}$ in $W^{1,2}_0(\Omega)$, (\ref{Ine}) yields
$$\int_\Omega |\nabla u_n|^2\;dx\leq C,$$
where $C>0$ is independent of $n$ but depends on $\tau$. Hence, there exists $u\in W^{1,2}_0(\Omega)$ such that $u_n\rightharpoonup u$ weakly in $W^{1,2}_0(\Omega)$ and $u_n\to u$ pointwise a.e. in $\Omega$. Assume that $\phi\in W^{1,2}_0({\Omega})\cap L^\infty(\Omega)$ is a compactly supported function, and we have 
\begin{align}
    \int_\Omega &A(x)\nabla u_n\cdot\nabla \phi \, dx+\int_{\mathbb{R}^N}\int_{\mathbb{R}^N}(u_n(x)-u_n(y))(\phi(x)-\phi(y))K(x,y)\, dx dy\nonumber\\
&=\int_\Omega\frac{-\De w_n \phi}{(|u_n|+\tau)^{\delta(x)}}\;dx+\int_\Omega \phi g\, dx.\nonumber\\
&=\int_\Omega \nabla w_n\cdot \nabla\left(\frac{\phi}{(|u_n|+\tau)^{\delta(x)}}\right)\;dx+\int_\Omega \phi g\, dx.
\end{align}
Taking $n\to\infty$ in both the sides, we obtain
\begin{align}\label{ges1}
      \int_\Omega &A(x)\nabla u\cdot\nabla \phi \, dx+\int_{\mathbb{R}^N}\int_{\mathbb{R}^N}(u(x)-u(y))(\phi(x)-\phi(y))K(x,y)\, dx dy\nonumber\\
     &=\lim_{n\to\infty}\int_\Omega \nabla w_n\cdot \nabla\left(\frac{\phi}{(|u_n|+\tau)^{\delta(x)}}\right)\;dx+\int_\Omega \phi g\, dx.
\end{align}
In order to pass the limit to the integral in the second last term in \eqref{ges1}, we observe
\begin{equation}\label{gesnew1}
\begin{split}
    \int_\Omega \nabla w_n\cdot\nabla\left (\frac{\phi}{(|u_n|+\tau)^{\delta(x)}}\right )\, dx&=\int_\Omega\frac{\nabla w_n\cdot\nabla \phi}{(|u_n|+\tau)^{\delta(x)}}\, dx-\int_\Omega\frac{\nabla w_n\cdot\nabla\delta(x)}{(|u_n|+\tau)^{\delta(x)}}\log{\Big(|u_n|+\tau\Big)}\phi\, dx\\
&-\int_\Omega\frac{\delta(x)\nabla w_n\cdot\nabla |u_n|}{(|u_n|+\tau)^{\delta(x)+1}}\phi\, dx.
\end{split}
\end{equation}
Since $w_n\to w$ strongly in $W^{1,2}_0(\Omega)$, and the sequences $\Big\{\frac{\nabla \phi}{(|u_n|+\tau)^{\delta(x)}}\Big\}_{n\in\N}$ converges weakly to the function $\frac{\nabla \phi}{(|u|+\tau)^{\delta(x)}}$ in $L^2(\Omega)^N$, we get
$$\lim_{n\to\infty}\int_\Omega\frac{\nabla w_n\cdot\nabla \phi}{(|u_n|+\tau)^{\delta(x)}}\, dx=\int_\Omega\frac{\nabla w\cdot\nabla \phi}{(|u|+\tau)^{\delta(x)}}\, dx.$$
Furthermore, using the facts (\ref{LC}), $|\nabla\delta|\in L^\infty_{\mathrm{loc}}(\Omega)$, we conclude that $$\frac{\nabla\delta(x)}{(|u_n|+\tau)^{\delta(x)}}\log{\big(|u_n|+\tau\big)}\phi\rightharpoonup\frac{\nabla\delta(x)}{(|u|+\tau)^{\delta(x)}}\log{\big(|u|+\tau\big)}\phi\text{ weakly in }L^2(\Omega)^N$$ and hence
$$\lim_{n\to\infty} \int_\Omega\frac{\nabla w_n\cdot\nabla\delta(x)}{(|u_n|+\tau)^{\delta(x)}}\log{\big(|u_n|+\tau\big)}\phi\, dx=\int_\Omega\frac{\nabla w\cdot\nabla\delta(x)}{(|u|+\tau)^{\delta(x)}}\log{\big(|u|+\tau\big)}\phi\, dx.$$
Moreover, since $\phi\in L^\infty(\Omega)$ and $|u_n|\rightharpoonup |u|$ weakly in $W^{1,2}_0(\Omega)$, by a similar argument, we arrive at
$$\lim_{n\to\infty}\int_\Omega\frac{\delta(x)\nabla w_n\cdot\nabla |u_n|}{(|u_n|+\tau)^{\delta(x)+1}}\phi\, dx=\int_\Omega\frac{\delta(x)\nabla w\cdot\nabla |u|}{(|u|+\tau)^{\delta(x)+1}}\phi\, dx.$$
Combining the above facts in \eqref{gesnew1}, we deduce
\begin{align}\label{ges}
    \lim_{n\to\infty}\int_\Omega \nabla w_n\cdot \nabla\left(\frac{\phi}{(|u_n|+\tau)^{\delta(x)}}\right)\;dx=\int_\Omega \nabla w\cdot \nabla\left(\frac{\phi}{(|u|+\tau)^{\delta(x)}}\right)\;dx.
\end{align}
Using (\ref{ges}) in (\ref{ges1}), we obtain
\begin{align}\label{Weak11}
    \int_\Omega &A(x)\nabla u\cdot\nabla \phi \, dx+\int_{\mathbb{R}^N}\int_{\mathbb{R}^N}(u(x)-u(y))(\phi(x)-\phi(y))K(x,y)\, dx dy\nonumber\\
     &=\int_\Omega \nabla w\cdot \nabla\left(\frac{\phi}{(|u|+\tau)^{\delta(x)}}\right)\;dx+\int_\Omega \phi g\, dx.
\end{align}
Now for a non-negative function $v\in W^{1,2}_0(\Omega)\cap L^\infty(\Omega)$, there exists as sequence of compactly supported function $\{v_n\}\subset W^{1,2}_0(\Omega)\cap L^\infty(\Omega)$ (see, \cite[Lemma A.1]{Hirano04}) such that $0\leq v_1\leq v_2\leq...\leq v_n\leq v_{n+1}\leq...$ and $v_n$ converges to $v$ strongly in $W^{1,2}_0(\Omega)$. Thus, putting $\phi=v_n$ in (\ref{Weak11}) and then taking $n\to\infty$, we deduce
\begin{align}\label{Weak2}
    \int_\Omega &A(x)\nabla u\cdot\nabla v \, dx+\int_{\mathbb{R}^N}\int_{\mathbb{R}^N}(u(x)-u(y))(v(x)-v(y))K(x,y)\, dx dy\nonumber\\
     &=\lim_{n\to\infty}\int_\Omega \nabla w\cdot \nabla\left(\frac{v_n}{(|u|+\tau)^{\delta(x)}}\right)\;dx+\int_\Omega v g\, dx.
\end{align}
In order to pass the limit to the integral, we use the equality
\begin{equation}\label{WS}
    \int_\Omega \nabla w\cdot \nabla\left(\frac{v_n}{(|u|+\tau)^{\delta(x)}}\right)\;dx=\int_\Omega \nabla w\cdot\left(\frac{\nabla v_n}{(|u|+\tau)^{\delta(x)}}\right)\;dx+\int_\Omega \nabla w\cdot \nabla\left(\frac{1}{(|u|+\tau)^{\delta(x)}}\right)v_n\;dx.
\end{equation}

Since $v_n\to v$ in $W^{1,2}_0(\Omega)$, we have
$$\lim_{n\to\infty}\int_\Omega \nabla w\cdot\left(\frac{\nabla v_n}{(|u|+\tau)^{\delta(x)}}\right)\;dx=\int_\Omega \nabla w\cdot\left(\frac{\nabla v}{(|u|+\tau)^{\delta(x)}}\right)\;dx.$$
Since $v\in L^\infty(\Omega)$, utilizing monotone convergence theorem along with generalized dominated convergence theorem, we get 
\begin{equation}\label{boundedness}
    \lim_{n\to\infty}\int_\Omega \nabla w\cdot \nabla\left(\frac{1}{(|u|+\tau)^{\delta(x)}}\right)v_n\;dx=\int_\Omega \nabla w\cdot \nabla\left(\frac{1}{(|u|+\tau)^{\delta(x)}}\right)v\;dx.
\end{equation}
Combining these two estimates along with \eqref{WS}, \eqref{Weak2} yields that for every non-negative $v\in W^{1,2}_0(\Omega)\cap L^\infty(\Omega)$ we have
\begin{align}\label{Weak}
    \int_\Omega &A(x)\nabla u\cdot\nabla v \, dx+\int_{\mathbb{R}^N}\int_{\mathbb{R}^N}(u(x)-u(y))(v(x)-v(y))K(x,y)\, dx dy\nonumber\\
     &=\int_\Omega \nabla w\cdot \nabla\left(\frac{v}{(|u|+\tau)^{\delta(x)}}\right)\;dx+\int_\Omega v g\, dx.
\end{align}
For every $v\in W^{1,2}_0(\Omega)\cap L^\infty(\Omega)$, we write $v=v^+-v^-$ and \eqref{Weak} hold for $v^+$ and $v^-$. Using linearity, we can conclude that \eqref{Weak} hold for $v$.\\
Now we shall prove that $u>0$ in $\Omega$. To this concern, we incorporate $v=-T^-_l(u)$ in \eqref{Weak} and use \eqref{lkernel}, \eqref{nkernel} to deduce
\begin{align}
    \alpha\int_\Omega |\nabla T^-_l(u)|^2\;dx\leq 0,\text{ for all }l>0.
\end{align}
Taking $l\to\infty$ and using Fatou's lemma, we get $u^-=0$ in $\Omega$. Consequently, $u\geq 0$ in $\Omega$. Thus, \cite[Lemma 4.6]{GKK} ensures $u>0$ in $\Omega$, which implies that $u$ is a weak solution of (\ref{APE0}).\\
\textbf{Uniqueness:} Let us assume that $u_1$ and $u_2$ are two such solutions of (\ref{APE01}). Thus,
$$\mathcal{M}(u_1-u_2)=\left(\frac{1}{(u_1+\tau)^{\delta(x)}}-\frac{1}{(u_2+\tau)^{\delta(x)}}\right)\Tilde{\nu} \text{ in }\Omega.$$
Incorporating $v= T^+_l(u_1-u_2)$ in the weak formulation of the above equation and utilizing (\ref{lkernel}) and (\ref{nkernel}), we obtain
\begin{align*}
    \alpha\int_\Omega |\nabla T^+_l(u_1-u_2)|^2\leq \int_\Omega\left(\frac{1}{(u_1+\tau)^{\delta(x)}}-\frac{1}{(u_2+\tau)^{\delta(x)}}\right)T^+_l(u_1-u_2)\;d\Tilde{\nu}\leq 0,
\end{align*}
which yields
$$\int_\Omega|T^+_l(u_1-u_2)|^{2^*}\;dx\leq 0.$$
By Fatou's lemma, we get
$$\int_\Omega|(u_1-u_2)^+|^{2^*}\;dx\leq 0,$$
which implies $u_1\leq u_2$ in $\Omega.$ Interchanging the roles of $u_1$ and $u_2$, we have $u_2\leq u_1$ in $\Omega$. Consequently, $u_1=u_2$ in $\Omega$. This completes the proof.
\end{proof}

We assume that $\{\nu_n\}_{n\in\mathbb{N}}\subset W^{-1,2}(\Omega)\setminus\{0\}$ is an increasing sequence of non-negative measures and $\{g_n\}_{n\in\mathbb{N}}\subset L^2(\Omega)$ is a sequence of non-negative functions in $\Omega$. Then, by Lemma \ref{Existence}, for each fixed $n\in\mathbb{N}$, there exists a unique $u_n\in W^{1,2}_0(\Omega)$ that solves the equation \ref{APE0} with $\tau=\frac{1}{n},\,\Tilde{\nu}=\nu_n,\,g=g_n$. That is, $u_n$ solves the following approximated problem:
\begin{align}\label{APE1}
    \begin{cases}
        \mathcal{M}u=\frac{\nu_n}{(u+\frac{1}{n})^{\delta(x)}}+g_n \text{ in }\Omega,\\
        u=0 \text{ in }\mathbb{R}^N\setminus\Omega \text{ and } u>0 \text{ in }\Omega,
    \end{cases}
\end{align}
in the sense of the integral identity (\ref{Weak}).

\begin{Lemma}\label{LD11}(Uniform positivity)
    Let for each $n\in\mathbb{N}$, $u_n\in W^{1,2}_0(\Omega)$ denotes the unique solution of the equation (\ref{APE1}). Then, for every $\omega\Subset\Omega$, there exists a constant $C(\omega)>0$ (independent of $n$) such that $u_n\geq C(\omega)$ in $\omega$ for all $n$.
\end{Lemma}
\begin{proof}
Let $n\in\mathbb{N}$. To obtain a uniform lower bound of $u_n$, we compare $u_n$ with the solution $w\in W^{1,2}_0(\Omega)$ of the equation \ref{APE0} with $\tau=1,\,\Tilde{\nu}=\nu_1,\,g=0$. That is $w$ satisfies the following equation:
\begin{align}\label{A77}
    \begin{cases}
         \mathcal{M}w=\frac{\nu_1}{(w+1)^{\delta(x)}}\text{ in }\Omega,\\
     w=0 \text{ in }\mathbb{R}^N\setminus\Omega \text{ and } w>0 \text{ in }\Omega.
    \end{cases}
\end{align}
By incorporating the test function $v=T^+_l(w-u_n)$ in the weak formulations of (\ref{APE1}),(\ref{A77}) and subtracting one from the other, we obtain
\begin{align*}
    \int_\Omega A(x)&\nabla(w-u_n)\cdot\nabla v\, dx+\int_{\mathbb{R}^N}\int_{\mathbb{R}^N}((w-u_n)(x)-(w-u_n)(y))(v(x)-v(y))K(x,y)dx\, dy\nonumber\\
    &=\int_\Omega\frac{v}{(w+1)^{\delta(x)}}\, d\nu_1
    -\int_\Omega\frac{v}{(u_n+\frac{1}{n})^{\delta(x)}}\, d\nu_n-\int_\Omega vg_n \, dx\nonumber\\
    &\leq \int_\Omega \left ( \frac{1}{(w+1)^{\delta(x)}}-\frac{1}{(u_n+\frac{1}{n})^{\delta(x)}}\right )T^+_l(w-u_n)\, d\nu_n\leq 0.
\end{align*}
In the last line, we used the fact $\nu_1\leq\nu_n$. Due to the property (\ref{lkernel}) and (\ref{nkernel}), we have 
\begin{align*}
    \alpha\int_\Omega|\nabla v|^2 dx+\underbrace{\Lambda^{-1}\int_{\mathbb{R}^N}\int_{\mathbb{R}^N}\frac{((w-u_n)(x)-(w-u_n)(y))(v(x)-v(y))}{|x-y|^{N+2s}}\, dxdy}_{\geq 0}\leq 0,
\end{align*}
which implies
\begin{align*}
    0\leq\int_{\Omega}|\nabla T^+_l(w-u_n)|^2\, dx\leq 0, \text{ for all }l>0.
\end{align*}
 Thus, $T^+_l(w-u_n)=0$ in $\Omega$ for all $l>0$, which implies $ (w-u_n)^+=0$ in $\Omega$. Hence, $u_n\geq w$ in $\Omega$. Since $n\in\mathbb{N}$ is arbitrary, $u_n\geq w$ for all $n\in\mathbb{N}$. {From \cite[Lemma 4.6]{GKK}, we can conclude that for every $\omega\Subset\Omega$, there exists a constant $C(\omega)>0$ such that $w\geq C(\omega)$ in $\omega$ and consequently, $u_n\geq C(\omega)$ in $\omega$ for all $n$.}
\end{proof}

\subsection{A priori estimates for {Theorem \ref{Theorem1}}}
The following uniform boundedness results will be useful for the proofs of Theorem \ref{Theorem1}.
\begin{Lemma}\label{LD2}(Uniform boundedness)
Assume that $\delta:\overline{\Omega}\to(0,\infty)$ is a continuous function such that $\delta$ is locally Lipschitz continuous in $\Omega$, $|\nabla\delta|\in L^N(\Omega)$, and $\delta$ satisfies $P_{\epsilon,\delta_*}$ for some $\;\delta_*\geq 1$ and for some $\epsilon>0$.

   Further, we suppose that the sequence of non-negative non-zero measures $\{\nu_n\}_{n\in\N}\subset W^{-1,2}(\Omega)\cap M^2_0(\Omega)$ is uniformly bounded in $\mathcal{M}(\Omega)$ (in total variation norm). Let $n\in\mathbb{N}$ and assume that the solution of (\ref{APE1}) obtained in Lemma \ref{Existence} is denoted by $u_n$. If the sequence of non-negative functions $\{g_n\}_{n\in\mathbb{N}}\subset L^2(\Omega)$ in $\Omega$ is uniformly bounded in $L^1(\Omega)$, then the following conclusions hold:
    \begin{enumerate}
        \item[(a)] If $\delta_*=1$, then the sequence $\{u_n\}_{n\in \mathbb{N}}$ is uniformly bounded in $W^{1,q}_0(\Omega)$ for every $1<q<\frac{N}{N-1}$.
        \item[(b)] If $\delta_*>1$, then the sequence $\{u_n\}_{n\in \mathbb{N}}$ is uniformly bounded in $W^{1,q}_{\mathrm{loc}}(\Omega)\cap L^1(\Omega)$ for every $1<q<\frac{N}{N-1}$. {Moreover, the sequences $\{G_k(u_n)\}_{n\in\mathbb{N}}$, $\{T_k(u_n)\}_{n\in\mathbb{N}}$ and $\Big\{T_k^\frac{\delta_*+1}{2}(u_n)\Big\}_{n\in\mathbb{N}}$ are uniformly bounded in $W^{1,q}_0(\Omega)$, $W^{1,2}_{\mathrm{loc}}(\Omega)$ and $W^{1,2}_0(\Omega)$, respectively for every fixed $k>0$ and $1<q<\frac{N}{N-1}$.}
    \end{enumerate}
\end{Lemma}
\begin{proof}
\begin{enumerate}
\item[$(a)$] We assume $\delta_*=1,$ which implies $0<\delta(x)\leq 1$ in $\Omega_\epsilon$. We shall prove that $\{u_n\}_{n\in\mathbb{N}}$ is uniformly bounded in $W^{1,q}_0(\Omega)$ for every $1<q<\frac{N}{N-1}$. Due to the fact (\ref{maremb}), it is sufficient to prove that $\{u_n\}_{n\in\mathbb{N}}$ is bounded in {$M^\frac{N}{N-1}(\Omega).$} Since $\Omega$ is bounded, it is enough to estimate the measure, $|x\in\Omega:\{|\nabla u_n|\geq t\}|$ for all $t\geq 1$. For any $l,t\geq 1$, we observe that
\begin{align}\label{KZ}
    |\{ x\in\Omega: |\nabla u_n|\geq t\}|\leq \underbrace{|\{x\in\Omega: |\nabla u_n|\geq t, u_n\leq l\}|}_{=I_2}+\underbrace{|\{x\in\Omega: |\nabla u_n|\geq t, u_n\geq l\}|}_{=I_1}.
\end{align}
We shall estimate $I_1$ and $I_2$. In order to estimate $I_1$, we put $\phi=T_l(u_n)$ as a test function in the weak formulation of the equation $(\ref{APE1})$ and use (\ref{lkernel}) and (\ref{nkernel}) to deduce
\begin{align}\label{AT9}
    \alpha\int_\Omega |\nabla T_l(u_n)|^2\, dx+&\Lambda^{-1}\underbrace{\int_{\mathbb{R}^N}\int_{\mathbb{R}^N} \frac{(u_n(x)-u_n(y))(T_l(u_n(x))-T_l(u_n(y)))}{|x-y|^{N+2s}}\, dxdy}_{\geq 0}\nonumber\\
    &\leq \int_\Omega\frac{T_l(u_n)}{(u_n+\frac{1}{n})^{\delta(x)}}\, d\nu_n+\int_\Omega T_l(u_n)g_n\, dx\nonumber\\
    &\leq \underbrace{\int_\Omega\frac{T_l(u_n)}{(u_n+\frac{1}{n})^{\delta(x)}}\, d\nu_n}_{=J_1}+l||g_n||_{L^1(\Omega)}.
\end{align}
We observe that, since $T_l(s)$ is an increasing function in $s$, the nonlocal integral above becomes non-negative. Using Lemma \ref{Bdd} and boundedness of $\{\nu_n\}_{n\in\N}$ in $\mathcal{M}(\Omega)$, we obtain 
\begin{align}
J_1&=\int_{\Omega}\frac{T_l(u_n)}{(u_n+\frac{1}{n})^{\delta(x)}}\, d\nu_n\leq \Big\|\frac{T_l(u_n)}{(u_n+\frac{1}{n})^{\delta(x)}}\Big\|_{L^\infty(\Omega)}|\nu_n(\Omega)|\leq C\Big\|\frac{T_l(u_n)}{(u_n+\frac{1}{n})^{\delta(x)}}\Big\|_{L^\infty(\Omega)} .
\end{align}
By using the property $P_{\epsilon,\delta_*}$ and Lemma \ref{LD11}, we deduce
\begin{align*}
    \Big|\frac{T_l(u_n)}{(u_n+\frac{1}{n})^{\delta(x)}}\Big|&=\Big|\frac{T_l(u_n)}{(u_n+\frac{1}{n})^{\delta(x)}}\chi_{\Omega\cap\Omega_\epsilon^c}(x)\Big|+\Big|\frac{T_l(u_n)}{(u_n+\frac{1}{n})^{\delta(x)}}\chi_{\Omega_\epsilon}(x)\Big|\\
    &\leq Cl+\Big|\frac{T_l(u_n)}{(u_n+\frac{1}{n})^{\delta(x)}}\chi_{\{x\in\Omega_\epsilon:u_n(x)>1\}}(x)\Big|+\Big|\frac{T_l(u_n)}{(u_n+\frac{1}{n})^{\delta(x)}}\chi_{\{x\in\Omega_\epsilon:u_n(x)\leq 1\}}(x)\Big|\\
    &\leq Cl+Cl+\Big|\frac{T_l(u_n)}{(T_l(u_n)+\frac{1}{n})^{\delta(x)}}\chi_{\{x\in\Omega_\epsilon:u_n(x)\leq 1\}}(x)\Big|\\
    &\leq 2Cl+T_l^{1-\delta(x)}(u_n)\chi_{\{x\in\Omega_\epsilon:u_n(x)\leq1\}}(x)\\
    &\leq 2Cl+l\leq Cl,
\end{align*}
which yields
\begin{align}\label{ATheorem10}
    J_1\leq Cl,
\end{align}
where $C>0$ is a constant independent of $n$. It follows from (\ref{AT9}) and (\ref{ATheorem10}) along with the uniform $L^1$ bound of $g_n$ that 
\begin{align}\label{ATheorem12}
    \int_\Omega |\nabla T_l(u_n)|^2\, dx\leq Cl,
\end{align}
where $C>0$ is a constant independent of $n$. By Lemma \ref{emb}, one has
\begin{align*}
    \left (\int_{\{x\in\Omega: u_n(x)\geq l\}} |T_l{(u_n)}|^{2^*}\, dx\right )^\frac{2}{2^*}\leq Cl,
\end{align*}
where {$2^*=\frac{2N}{N-2}$}. Here, $C>0$ is a constant independent of $n$. This yields 
\begin{align*}
    |\{x\in\Omega: u_n(x)\geq l\}|^\frac{2}{2^*}\leq \frac{C}{l}.
\end{align*}
Thus, there exists a constant $C>0$ is a constant independent of $n$ such that
\begin{align}\label{ST2}
   I_1\leq \left |\{x\in\Omega: u_n(x)\geq l\}\right |\leq \frac{C}{l^{\frac{N}{N-2}}}, \text{ for all }l\geq 1.
\end{align}
Now we estimate $I_2$. To this concern, we deduce
\begin{align}\label{ATheorem14}
    I_2&=|\{x\in\Omega: |\nabla u_n|\geq t, u_n\leq l\}|\leq \frac{1}{t^2}\int_{\{x\in\Omega:u_n\leq l\}}|\nabla u_n|^2\, dx\nonumber\\
    &\leq \frac{1}{t^2}\int_\Omega |\nabla T_l(u_n)|^2\, dx\leq \frac{Cl}{t^2},
\end{align}
where $C>0$ is a constant independent of $n$. The rest of the proof follows by proceeding along the lines of the proof of \cite[Lemma A.2, page 23]{BG}.

\item[$(b)$] We suppose that $\delta_*>1$. We shall prove that for every fixed $k>0$, the sequence $\{G_k(u_n)\}_{n\in\mathbb{N}}$ is uniformly bounded in $W^{1,q}_0(\Omega)$ for every $1<q<\frac{N}{N-1}.$

To this end, similar to part $(a)$ above, we estimate the measure $|\{x\in\Omega:|\nabla G_k(u_n)|\geq t\}|$ for all $t\geq 1.$ Let $l>0$. Then, we observe that 
\begin{align}\label{ineq1}
    &|\{x\in\Omega:\nabla G_k(u_n)|\geq t\}|\nonumber\\
    &\leq \underbrace{|\{x\in\Omega:|\nabla G_k(u_n)|\geq t, G_k(u_n)\leq l\}}_{I_2}|+\underbrace{|\{x\in\Omega:|\nabla G_k(u_n)|\geq t, G_k(u_n)\geq l\}|}_{I_1}.
\end{align}
We estimate $I_1$ and $I_2$ using the same approach as in part $(a)$ above. By taking the test function $\phi=T_l(G_k(u_n))$ in the weak formulation of (\ref{APE1}) and applying Lemma \ref{Bdd} along with uniform bound of $\{\nu_n\}$ in $\mathcal{M}(\Omega)$, we obtain
\begin{align}\label{cl}
    \int_\Omega A(x)\nabla u_n\cdot \nabla T_l(G_k(u_n)) \, dx&+\underbrace{\int_{\mathbb{R}^N}\int_{\mathbb{R}^N}(u_n(x)-u_n(y))(\phi(x)-\phi(y)K(x,y)\, dxdy}_{\geq 0}\nonumber\\
    &=\int_\Omega \frac{T_l(G_k(u_n))}{(u_n+\frac{1}{n})^{\delta(x)}}\, d\nu_n+\int_\Omega g_nT_l(G_k(u_n))\, dx\nonumber\\
    &\leq l\left\lVert \frac{1}{k^{\delta(x)}}\right\rVert_{L^\infty(\Omega)}\nu_n(\Omega)+l\int_\Omega g_n\leq Cl,
\end{align}
where the constant $C>0$ is independent of $n.$ We observe that, since $T_l(s)$ is an increasing function in $s$, the nonlocal integral above becomes non-negative. We have used the fact that $u_n\geq k$ in $\mathrm{supp}\,T_l(G_k(u_n))$ to obtain the first inequality in \eqref{cl} above. Then the above inequality \eqref{cl} together with (\ref{lkernel}) leads to
\begin{align}\label{W4}
  \int_\Omega |\nabla T_l(G_k(u_n))|^2\, dx\leq Cl,
\end{align}
where $C>0$ is a constant independent of $n.$ The rest of the proof follows similar to that of claim 1 in \cite[Lemma A.2, page 24]{BG}.\\
 We are left with the proof of the facts that the sequences $\{T_k(u_n)\}_{n\in\mathbb{N}}$ and $\Big\{T_k(u_n)^\frac{\delta_*+1}{2}\Big\}_{n\in\mathbb{N}}$ are uniformly bounded in $W^{1,2}_{\mathrm{loc}}(\Omega)$ and $W^{1,2}_{0}(\Omega)$, respectively for every fixed $k>0$.

To this concern, we put $\phi=T^{\delta_*}_k(u_n)$ into the weak formulation of (\ref{APE1}) and use of the properties (\ref{lkernel}) and (\ref{nkernel}) to deduce
\begin{align}\label{W2}
    \frac{4\alpha\delta_*}{(\delta_*+1)^2}\int_\Omega \Big|\nabla T^\frac{\delta_*+1}{2}_k(u_n)\Big|^2\, dx&+\Lambda^{-1}\underbrace{\int_{\mathbb{R}^N}\int_{\mathbb{R}^N} \frac{(u_n(x)-u_n(y))(\phi(x)-\phi(y))}{|x-y|^{N+2s}}\, dxdy}_{\geq 0}\nonumber\\
    &\leq \int_\Omega \frac{T^{\delta_*}_k(u_n)}{(u_n+\frac{1}{n})^{\delta(x)}}\, d\nu_n+\int_\Omega T^{\delta_*}_k(u_n) g_n\, dx.
\end{align}
Using the uniform $L^1$ bound of $g_n$, and uniform bound of the total variation of $\{\nu_n\}$ along with Lemma \ref{Bdd}, (\ref{W2}) yields
\begin{align}\label{TB}
    \int_\Omega \Big|\nabla T^\frac{\delta_*+1}{2}_k(u_n)\Big|^2\, dx&\leq C\Big(  \|T^{\delta_*-\delta(x)}_k(u_n)\|_{L^\infty(\Omega)}|\nu_n(\Omega)|+k^{\delta_*}\int_\Omega g_n\, dx\Big)\nonumber\\
    &\leq C\Big(\|k^{\delta_*-\delta(x)}\|_{L^\infty(\Omega)}+k^{\delta_*}\Big)\leq C
\end{align}
where $C>0$ is a constant independent of $n$. Hence, the sequence $\Big\{T_k(u_n)^\frac{\delta_*+1}{2}\Big\}_{n\in\mathbb{N}}$ is uniformly bounded in $W^{1,2}_{0}(\Omega)$ for every fixed $k>0$.The rest of the proof follows similarly to \cite[Lemma A.2, page 26]{BG}.
\end{enumerate}
\end{proof}

\subsection{A priori estimates for Theorem \ref{Theorem2}}
\begin{Lemma}\label{BL2}
Assume that the function $\delta:\overline{\Omega}\to (0,\infty)$ is a constant function. Let $n\in\mathbb{N}$ and define $g_n(x)=T_n(g(x))=\min\{g(x),n\}$ {in \eqref{APE1}}, where $g\in L^{\frac{N(\delta+1)}{N+2\delta}}(\Omega)$ is a non-negative function in $\Omega$. {Suppose that the solution of \eqref{APE1} obtained in Lemma \ref{Existence} is denoted by $u_n$.} If the sequence of non-negative non-zero measures $\{\nu_n\}_{n\in\N}\subset W^{-1,2}(\Omega)\cap M^2_0(\Omega)$ is uniformly bounded in $\mathcal{M}(\Omega)$ (in total variation norm), then the following conclusions hold:
\begin{enumerate}
    \item[$(a)$] If $\delta=1$, then $\{u_n\}_{n\in\mathbb{N}}$ is uniformly bounded in $W^{1,2}_0(\Omega)$.
    \item[$(b)$] If $0<\delta<1$, then $\{u_n\}_{n\in\mathbb{N}}$ is uniformly bounded in $W^{1,q}_0(\Omega)$, where $q=\frac{N(\delta+1)}{N+\delta-1}$.
     \item[$(c)$] If $\delta>1$, then $\{u_n\}_{n\in\mathbb{N}}$ is uniformly bounded in $W^{1,2}_{\mathrm{loc}}(\Omega)$. Moreover, $\{u^\frac{\delta+1}{2}_n\}_{n\in\mathbb{N}}$ is uniformly bounded in $W^{1,2}_0(\Omega).$
\end{enumerate}
\end{Lemma}

\begin{proof}
\begin{enumerate}
\item[$(a)$] Choosing $v=T_l(u_n)$ as a test function in the weak formulation of (\ref{APE1}) and applying the properties (\ref{lkernel}), (\ref{nkernel}), we get
    \begin{align}\label{RR2}
        \alpha\int_\Omega |\nabla T_l(u_n)|^2\, dx+&\Lambda^{-1}\underbrace{\int_{\mathbb{R}^N}\int_{\mathbb{R}^N}\frac{(u_n(x)-u_n(y))(T_l(u_n)(x)-T_l(u_n)(y))}{|x-y|^{N+2s}}\, dxdy}_{\geq 0}\nonumber\\
        &\leq\int_\Omega\frac{T_l(u_n)}{(u_n+\frac{1}{n})}\, d\nu_n+\int_\Omega g_nT_l(u_n)\, dx.
    \end{align}
    Using H\"{o}lder's inequality, $\nu_n$ is bounded in $\mathcal{M}(\Omega)$ along with Lemma \ref{emb} and Lemma \ref{Bdd} in (\ref{RR2}), it follows that 
\begin{align}
    \alpha||T_l(u_n)||_{{W_0^{1,2}(\Omega)}}^2\leq \nu_n(\Omega)+\int gT_l(u_n)\, dx&\leq C+||g||_{L^{(2^*)'}(\Omega)}||T_l(u_n)||_{L^{2^*}(\Omega)}\nonumber\\
    &\leq C(1+||T_l(u_n)||_{{W_0^{1,2}(\Omega)}}),
\end{align}
where $C$ is a constant independent of $n,\,l$. Hence, $||T_l(u_n)||_{{W_0^{1,2}(\Omega)}}^2\leq C,$
where $C$ is independent of $n,l$. Taking $l\to\infty$ and using Fatou's lemma, we conclude that the sequence $\{u_n\}_{n\in\mathbb{N}}$ is uniformly bounded in $W^{1,2}_0(\Omega).$

\item[$(b)$] If $0<\delta<1$, then for $0<\epsilon<\frac{1}{n}$, choosing $\phi=(T_l(u_n)+\epsilon)^\delta-\epsilon^\delta$ as a test function in the weak formulation of (\ref{APE1}) and using the properties (\ref{lkernel}), (\ref{nkernel}), we obtain
\begin{align}\label{RR1}
     \alpha\int_\Omega\Big|\nabla (T_l(u_n)+\epsilon)^{\frac{\delta+1}{2}}\Big|^2\, dx+&\Lambda^{-1}\underbrace{\int_{\mathbb{R}^N}\int_{\mathbb{R}^N}\frac{(u_n(x)-u_n(y))(\phi(x)-\phi(y))}{|x-y|^{N+2s}}\, dxdy}_{\geq 0}\nonumber\\
     &\leq \int_\Omega\frac{\phi}{(u_n+\frac{1}{n})^\delta}\, d\nu_n+\int_\Omega g_n{\phi}\, dx\nonumber\\
     &\leq \int_\Omega \frac{(T_l(u_n)+\epsilon)^\delta}{(T_l(u_n)+\frac{1}{n})^\delta}\, d\nu_n+\int_\Omega g(T_l(u_n)+\epsilon)^\delta\, dx.
\end{align}
Using H\"{o}lder inequality, $\nu_n$ is bounded in $\mathcal{M}(\Omega)$ and Lemma \ref{Bdd}, (\ref{RR1}) yields
\begin{align}\label{BW}
    \int_\Omega\Big|\nabla (T_l(u_n)+\epsilon)^{\frac{\delta+1}{2}}\Big|^2\, dx&\leq \left [|\nu_n(\Omega)|+\left (\int_\Omega (T_l(u_n)+\epsilon)^\frac{2^*(\delta+1)}{2}\, dx\right )^\frac{2\delta}{2^*(\delta+1)}||g||_{L^r(\Omega)}\right ]\nonumber\\
    &\leq C\Big[1+\left (\int_\Omega (T_l(u_n)+\epsilon)^\frac{2^*(\delta+1)}{2}\, dx\right )^\frac{2\delta}{2^*(\delta+1)}\Big],
\end{align}
where $r=\frac{N(\delta+1)}{N+2\delta}$ and $2^*=\frac{2N}{N-2}$. Arguing as in \cite[Lemma A.3-(b)]{BG}, we derive
$$\int_\Omega|\nabla T_l(u_n)|^q\, dx\leq C,$$
where $q=\frac{N(\delta+1)}{N+\delta-1}$ and $C>0$ is independent of $n,\;l$. Taking $l\to\infty$ and applying Fatou's lemma, we obtain the result.

\item[$(c)$] Let $k>0$, then we choose $T_k^\delta(u_n)$ as a test function in the weak formulation of the equation (\ref{BP}) with $\delta>1$ and apply (\ref{lkernel}), (\ref{nkernel}) to obtain
\begin{align*}
\alpha\delta\int_\Omega\Big|\nabla T_k^\frac{\delta+1}{2}(u_n)\Big|^2\, dx+&\Lambda^{-1}\underbrace{\int_{\mathbb{R}^N}\int_{\mathbb{R}^N}\frac{(u_n(x)-u_n(y))(T_k^\delta(u_n(x))-T_k^\delta(u_n(y)))}{|x-y|^{N+2s}}\, dxdy}_{\geq 0}\\
    &\leq \int_\Omega \frac{T_k^\delta(u_n)}{(u_n+\frac{1}{n})^\delta}\, d\nu_n\nonumber+\int_\Omega g_nT_k^\delta(u_n)\, dx.
\end{align*}
Using the fact $T_k(s)\leq s$ for every $k,s>0$ and Lemma \ref{Bdd}, the above identity yields
\begin{align}
    \int_\Omega\Big|\nabla T_k^\frac{\delta+1}{2}(u_n)\Big|^2\, dx\leq \nu_n(\Omega)+\int_\Omega gT_k^\delta(u_n)\, dx\leq C+\int_\Omega gT_k^\delta(u_n)\, dx,
\end{align}
 where $C$ is an upper bound of the sequence $\{|\nu_n(\Omega)|\}$. The remainder of the proof proceeds similarly to the argument used in \cite[Lemma A.3, page 29]{BG}.
\end{enumerate}
\end{proof}
\section*{Acknowledgement} P. Garain thanks IISER Berhampur for the Seed grant: IISERBPR/RD/OO/2024/15, Date: February 08, 2024.


\begin{thebibliography}{10}

\bibitem{Alvesjde}
Claudianor~O. Alves, Carlos~Alberto Santos, and Thiago~Willians Siqueira.
\newblock Uniqueness in {$W_{loc}^{1, p(x)}(\Omega)$} and continuity up to portions of the boundary of positive solutions for a strongly-singular elliptic problem.
\newblock {\em J. Differential Equations}, 269(12):11279--11327, 2020.

\bibitem{GG}
Gurdev~Chand Anthal and Prashanta Garain.
\newblock Symmetry, existence and regularity results for a class of mixed local-nonlocal semilinear singular elliptic problem via variational characterization, 2024 (to appear in Math Z).

\bibitem{Arcoyadie}
David Arcoya and Lucio Boccardo.
\newblock Multiplicity of solutions for a {D}irichlet problem with a singular and a supercritical nonlinearities.
\newblock {\em Differential Integral Equations}, 26(1-2):119--128, 2013.

\bibitem{Arcoyana}
David Arcoya and Lourdes Moreno-M\'erida.
\newblock Multiplicity of solutions for a {D}irichlet problem with a strongly singular nonlinearity.
\newblock {\em Nonlinear Anal.}, 95:281--291, 2014.

\bibitem{Arora}
Rakesh {Arora} and Vicentiu~D. {Radulescu}.
\newblock {Combined effects in mixed local-nonlocal stationary problems}.
\newblock {\em Proceedings of the Royal Society of Edinburgh: Section A Mathematics. Published online 2023:1-47}, 2023.

\bibitem{BalDas}
Kaushik {Bal} and Stuti {Das}.
\newblock {Multiplicity of solutions for mixed local-nonlocal elliptic equations with singular nonlinearity}.
\newblock {\em arXiv e-prints}, page arXiv:2405.05832, May 2024.

\bibitem{BGM}
Kaushik Bal, Prashanta Garain, and Tuhina Mukherjee.
\newblock On an anisotropic {$p$}-{L}aplace equation with variable singular exponent.
\newblock {\em Adv. Differential Equations}, 26(11-12):535--562, 2021.

\bibitem{meapp}
P.~Baras and M.~Pierre.
\newblock Singularit\'es \'eliminables pour des \'equations semi-lin\'eaires.
\newblock {\em Ann. Inst. Fourier (Grenoble)}, 34(1):185--206, 1984.

\bibitem{Ghosh}
Souvik {Bhowmick} and Sekhar {Ghosh}.
\newblock {Existence results for mixed local and nonlocal elliptic equations involving singularity and nonregular data}.
\newblock {\em arXiv e-prints}, page arXiv:2410.04441, October 2024.

\bibitem{Vecchihong}
Stefano Biagi, Serena Dipierro, Enrico Valdinoci, and Eugenio Vecchi.
\newblock A {H}ong-{K}rahn-{S}zeg\"o{} inequality for mixed local and nonlocal operators.
\newblock {\em Math. Eng.}, 5(1):Paper No. 014, 25, 2023.

\bibitem{VecchiBO}
Stefano Biagi, Dimitri Mugnai, and Eugenio Vecchi.
\newblock A {B}rezis-{O}swald approach for mixed local and nonlocal operators.
\newblock {\em Commun. Contemp. Math.}, 26(2):Paper No. 2250057, 28, 2024.

\bibitem{Vecchicvpde}
Stefano Biagi and Eugenio Vecchi.
\newblock Multiplicity of positive solutions for mixed local-nonlocal singular critical problems.
\newblock {\em Calc. Var. Partial Differential Equations}, 63(9):Paper No. 221, 45, 2024.

\bibitem{Biroud}
Kheireddine Biroud.
\newblock Mixed local and nonlocal equation with singular nonlinearity having variable exponent.
\newblock {\em J. Pseudo-Differ. Oper. Appl.}, 14(1):Paper No. 13, 24, 2023.

\bibitem{biswas2025}
Sanjit Biswas.
\newblock Multiplicity results for mixed local-nonlocal equations with singular and critical exponential nonlinearity in $\mathbb{R}^2$, 2025.

\bibitem{BG}
Sanjit Biswas and Prashanta Garain.
\newblock Regularity and existence for semilinear mixed local–nonlocal equations with variable singularities and measure data.
\newblock {\em Communications in Contemporary Mathematics}, 0(0):2550028, 0.

\bibitem{GOB}
Lucio Boccardo, Thierry Gallou\"et, and Luigi Orsina.
\newblock Existence and uniqueness of entropy solutions for nonlinear elliptic equations with measure data.
\newblock {\em Ann. Inst. H. Poincar\'e{} C Anal. Non Lin\'eaire}, 13(5):539--551, 1996.

\bibitem{BocOrs}
Lucio Boccardo and Luigi Orsina.
\newblock Semilinear elliptic equations with singular nonlinearities.
\newblock {\em Calc. Var. Partial Differential Equations}, 37(3-4):363--380, 2010.

\bibitem{Silva}
S.~Buccheri, J.~V. da~Silva, and L.~H. de~Miranda.
\newblock A system of local/nonlocal {{\(p\)}}-laplacians: the eigenvalue problem and its asymptotic limit as {{\(p\rightarrow \infty\)}}.
\newblock {\em Asymptotic Anal.}, 128(2):149--181, 2022.

\bibitem{Sciunzi}
Annamaria Canino, Luigi Montoro, Berardino Sciunzi, and Marco Squassina.
\newblock Nonlocal problems with singular nonlinearity.
\newblock {\em Bull. Sci. Math.}, 141(3):223--250, 2017.

\bibitem{CMvar}
Jos\'e{} Carmona and Pedro~J. Mart\'inez-Aparicio.
\newblock A singular semilinear elliptic equation with a variable exponent.
\newblock {\em Adv. Nonlinear Stud.}, 16(3):491--498, 2016.

\bibitem{MC18}
Michel Chipot and Linda~Maria De~Cave.
\newblock New techniques for solving some class of singular elliptic equations.
\newblock {\em Atti Accad. Naz. Lincei Rend. Lincei Mat. Appl.}, 29(3):487--510, 2018.

\bibitem{CRT}
M.~G. Crandall, P.~H. Rabinowitz, and L.~Tartar.
\newblock On a {D}irichlet problem with a singular nonlinearity.
\newblock {\em Comm. Partial Differential Equations}, 2(2):193--222, 1977.

\bibitem{DAL}
Gianni Dal~Maso.
\newblock On the integral representation of certain local functionals.
\newblock {\em Ricerche Mat.}, 32(1):85--113, 1983.

\bibitem{MD}
Gianni Dal~Maso, Fran\c~cois Murat, Luigi Orsina, and Alain Prignet.
\newblock Renormalized solutions of elliptic equations with general measure data.
\newblock {\em Ann. Scuola Norm. Sup. Pisa Cl. Sci. (4)}, 28(4):741--808, 1999.

\bibitem{Hitchhikersguide}
Eleonora Di~Nezza, Giampiero Palatucci, and Enrico Valdinoci.
\newblock Hitchhiker's guide to the fractional {Sobolev} spaces.
\newblock {\em Bull. Sci. Math.}, 136(5):521--573, 2012.

\bibitem{LC}
Lawrence~C Evans.
\newblock {\em Partial differential equations}, volume~19.
\newblock American Mathematical Society, 2022.

\bibitem{Fang}
Yanqin {Fang}.
\newblock {Existence, Uniqueness of Positive Solution to a Fractional Laplacians with Singular Nonlinearity}.
\newblock {\em arXiv e-prints}, page arXiv:1403.3149, March 2014.

\bibitem{Gjgea}
Prashanta Garain.
\newblock On a class of mixed local and nonlocal semilinear elliptic equation with singular nonlinearity.
\newblock {\em J. Geom. Anal.}, 33(7):Paper No. 212, 20, 2023.

\bibitem{PGjms}
Prashanta Garain.
\newblock Mixed anisotropic and nonlocal {Sobolev} type inequalities with extremal.
\newblock {\em J. Math. Sci., New York}, 281(5):633--645, 2024.

\bibitem{Garainmm}
Prashanta Garain.
\newblock On the regularity and existence of weak solutions for a class of degenerate singular elliptic problem.
\newblock {\em Manuscripta Math.}, 174(1-2):141--158, 2024.

\bibitem{GKK}
Prashanta Garain, Wontae Kim, and Juha Kinnunen.
\newblock On the regularity theory for mixed anisotropic and nonlocal {$p$}-{L}aplace equations and its applications to singular problems.
\newblock {\em Forum Math.}, 36(3):697--715, 2024.

\bibitem{GMcpaa}
Prashanta Garain and Tuhina Mukherjee.
\newblock Quasilinear nonlocal elliptic problems with variable singular exponent.
\newblock {\em Commun. Pure Appl. Anal.}, 19(11):5059--5075, 2020.

\bibitem{Guna}
Prashanta Garain and Alexander Ukhlov.
\newblock Mixed local and nonlocal {S}obolev inequalities with extremal and associated quasilinear singular elliptic problems.
\newblock {\em Nonlinear Anal.}, 223:Paper No. 113022, 35, 2022.

\bibitem{Giri}
Sekhar Ghosh, Debajyoti Choudhuri, and Ratan~Kr. Giri.
\newblock Singular nonlocal problem involving measure data.
\newblock {\em Bull. Braz. Math. Soc. (N.S.)}, 50(1):187--209, 2019.

\bibitem{Heinonen}
Juha Heinonen, Tero Kilpel\"ainen, and Olli Martio.
\newblock {\em Nonlinear potential theory of degenerate elliptic equations}.
\newblock Dover Publications, Inc., Mineola, NY, 2006.
\newblock Unabridged republication of the 1993 original.

\bibitem{Hirano04}
Norimichi Hirano, Claudio Saccon, and Naoki Shioji.
\newblock Existence of multiple positive solutions for singular elliptic problems with concave and convex nonlinearities.
\newblock {\em Adv. Differential Equations}, 9(1-2):197--220, 2004.

\bibitem{Hichem}
Shuibo Huang and Hichem Hajaiej.
\newblock Lazer-mckenna type problem involving mixed local and nonlocal elliptic operators.
\newblock {\em Nonlinear Differential Equations and Applications}, 03 2023.

\bibitem{JMar}
Joseph Marcinkiewicz.
\newblock Sur l'interpolation d'op{\'e}rations.
\newblock {\em C. R. Acad. Sci., Paris}, 208:1272--1273, 1939.

\bibitem{POesaim}
Francescantonio Oliva and Francesco Petitta.
\newblock On singular elliptic equations with measure sources.
\newblock {\em ESAIM Control Optim. Calc. Var.}, 22(1):289--308, 2016.

\bibitem{OLP}
Francescantonio Oliva and Francesco Petitta.
\newblock Singular elliptic pdes: an extensive overview, 2024.

\bibitem{OPdie}
Luigi Orsina and Francesco Petitta.
\newblock A {L}azer-{M}c{K}enna type problem with measures.
\newblock {\em Differential Integral Equations}, 29(1-2):19--36, 2016.

\bibitem{PKvar}
Ambesh~Kumar Pandey and Rasmita Kar.
\newblock Existence of solutions to semilinear elliptic equations with variable exponent singularity.
\newblock {\em Rocky Mountain J. Math.}, 53(3):903--913, 2023.

\bibitem{Royden}
Halsey Royden and Patrick~Michael Fitzpatrick.
\newblock {\em Real analysis}.
\newblock China Machine Press, 2010.

\bibitem{Vecchihenon}
Ariel~M. Salort and Eugenio Vecchi.
\newblock On the mixed local-nonlocal {H{\'e}non} equation.
\newblock {\em Differ. Integral Equ.}, 35(11-12):795--818, 2022.

\end{thebibliography}
\end{document}